\documentclass[a4paper]{article}
\usepackage[latin1]{inputenc}
\usepackage[T1]{fontenc}
\usepackage{amsmath,amscd}
\usepackage{graphicx}
\usepackage{amsthm}
\usepackage{amssymb}
\usepackage{bbm}
\usepackage{ae}
\usepackage[all]{xy}

\newtheorem{sats}{Theorem}[section]
\newtheorem{deef}[sats]{Definition}
\newtheorem{lem}[sats]{Lemma}

\newtheorem{prop}[sats]{Proposition}

\newcommand{\R}{\mathbbm{R}}
\newcommand{\C}{\mathbbm{C}}

\newcommand{\Z}{\mathbbm{Z}}

\newcommand{\ellL}{\mathcal{L}}

\newcommand{\id}{ \mathrm{id}}

\newcommand{\Ko}{\mathcal{K}}
\newcommand{\Bo}{\mathcal{B}}

\newcommand{\T}{\mathbbm{T}}

\newcommand{\Fg}{\mathcal{F}}

\renewcommand{\epsilon}{\varepsilon}
\renewcommand{\phi}{\varphi}

\newcommand{\im}{\mathrm{i} \mathrm{m} \,}

\title{
A remark on twists and the notion of torsion-free discrete quantum groups}
\author{Magnus Goffeng}
\date{Department of Mathematical Sciences, Division of Mathematics\\
Chalmers university of Technology and University of Gothenburg}

\begin{document}
\maketitle

\begin{abstract}
In this paper twists of reduced locally compact quantum groups are studied. Twists of the dual coaction on a reduced crossed product are introduced and the twisted dual coactions are proved to satisfy a type of Takesaki-Takai duality. The twisted Takesaki-Takai duality implies that twists of discrete, torsion-free quantum groups are torsion-free. Cocycle twists of duals of semisimple, compact Lie are studied leading to a locally compact quantum group contained in the Drinfeld-Jimbo algebra which gives a dual notion of Woronowicz deformations for semisimple, compact Lie groups. These cocycle twists are proven to be torsion-free whenever the Lie group is simply connected.
\end{abstract}

\section*{Introduction}

Quantum groups have long been studied as a natural generalization of groups. On the algebraic level they became interresting due to their applications in condensed matter theory and their relation with the algebraic Bethe ansatz. To the operator algebraists, the algebraic formalism of quantum groups lead to Kac algebras which gave a good setting to generalize Pontryagin duals to non-abelian groups, see more in \cite{enva}. In \cite{woro}, Woronowicz introduced a compact, non-commutative and non-cocommutative bi-$C^*$-algebra $SU_q(2)$, containing many of the group like structures that were found in Kac algebras, except that $SU_q(2)$ did not admit a tracial Haar weight. The discovery of these, more esoteric, quantum groups leads to the notion of multiplicative unitaries in \cite{baskauni}. A seemingly more intrinsic, but actually more restrictive, notion of reduced quantum groups was introduced in \cite{vaku}. The reduced quantum groups are the locally compact quantum groups admitting faithful Haar weights. The reduced setting even allows for Pontryagin duality. 

In the case of Lie groups two different approaches, the algebraic versus the operator algebraic, were put in a context in \cite{rosso} where the $C^*$-algebraic deformed Lie groups of Woronowicz were shown to be algebraically dual to the Drinfeld-Jimbo algebras which are cocycle twists of universal enveloping Lie algebras. In the literature the compact setting is usually favorized due to the existence of Haar states on compact quantum groups. In \cite{deco} twists of von Neumann quantum groups were studied and it was shown that a twisted quantum group is also a von Neumann algebraic quantum group in the sense of \cite{vakutv}. 

The Baum-Connes property of  a group is an interesting property with deep implications in representation theory. The study of the Baum-Connes conjecture as a homological property of the equivariant $KK^G$-category in \cite{neme} allows for a generalization of the Baum-Connes conjecture beyond groups to certain quantum groups. One generalization was introduced in \cite{mehom} for discrete quantum groups with a property generalizing torsion-freeness of discrete groups. These two conditions were needed to simplify the homological algebra. In the paper \cite{mehom} the question was posed whether the discrete duals of the Woronowicz deformations are torsion-free in this sense? The question is motivated by the fact that duals of compact groups satisfying the Hodgkin condition, such as $SU(n)$, are torsion-free in the sense of Meyer. \\

This paper deals with twists and cocycle twists of reduced locally compact quantum groups. In particular, we will study how a twist affects torsion-free discrete quantum groups. The main theme is the following property of a cocycle twist; it does not give rise to torsion in a discrete quantum group. We prove this property for twists and for the cocycle twist that defines the quantum group lying inside the Drinfeld-Jimbo algebra of a simply connected Lie group.

We will in the first section of this paper give some preliminaries on locally compact quantum groups, mainly from the paper \cite{vaku}. In the second section we recall the notion of a twist and cocycle twist for reduced locally compact quantum groups (see Definition \ref{twistdef}). In the von Neumann setting, twists were studied in \cite{deco}. We recall the result from \cite{deco} that twisting von Neumann algebraic quantum groups produces von Neumann algebraic quantum groups. Under the condition that our twists are \emph{feasible}, see Definition \ref{feasi}, we can twist the dual action on a crossed product to a coaction of the twisted quantum group, see Proposition \ref{contcoatw}. The feasibility condition is always satisfied by twists of discrete quantum groups. In Theorem \ref{tata} we show that, under a regularity condition on the twist, the twisted coaction satisfy a type of Takesaki-Takai duality in the sense that the crossed product of a twisted coaction is Morita equivalent to the crossed product before twisting.

In the third section we will study the twisted Takesaki-Takai duality in a more precise way. In Lemma \ref{twotimmor} we show that the Takesaki-Takai duality gives back the same $C^*$-algebra one started with, except for a new coaction of the twisted quantum group. This result is used to prove that if a discrete quantum group is torsion-free, then so are all the twists. However, the question posed in \cite{mehom} whether the duals of the Woronowicz deformations are torsion-free cannot be answered by these methods since the Drinfeld-Jimbo twist is not a twist. 

In the final section we study the concept of a Drinfeld-Jimbo algebra and the locally compact quantum group lying inside it using the ideas of \cite{netu} to describe the Drinfeld-Jimbo algebra as a cocycle twist of the dual of a semisimple, compact Lie group. By \cite{rosso}, the simply-connected case produces the dual of Woronowicz's deformed Lie group. The idea here lies in using weight theory for the Lie group to embed the dual into the dual of the simple connected covering and the twisting is induced from that on the simple connected covering. For the simply connected Lie groups, we show that the locally compact quantum group lying inside the Drinfeld-Jimbo algebra are in fact torsion-free answering the question posed by Meyer in \cite{mehom}. This proof uses the case when the group is $SU(2)$, which was studied in \cite{voigtfree}.

\section{Preliminaries}
We will start the paper by setting notations and recalling some useful results from the paper \cite{vaku} where the notion of a reduced locally compact quantum group was introduced. The letter $S$ will denote a bi-$C^*$-algebra with comultiplication $\Delta:S\to \mathcal{M}(S\otimes S)$ satisfying the coassociativity condition 
\[(\Delta\otimes \id)\Delta=(\id\otimes\Delta)\Delta.\]
The tensor products we will use throughout the paper is the minimal tensor product for $C^*$-algebras which will be denoted by $\otimes$ and we will denote the von Neumann tensor product by $\overline{\otimes}$. As usual, $\mathcal{M}(A)$ denotes the multiplier algebra of $A$. A weight $\phi$ is called proper if it is non-zero, densely defined and lower semicontinuous. A proper weight $\phi$ on $S$ is said to be left invariant if it satisfies 
\[\phi((\omega \otimes \id(\Delta(a)))=\omega(1)\phi(a)\] 
for all positive $a$ such that $\phi(a)<+\infty$ and $\omega\in S^*_+$. The definition of a left invariant proper weight makes sense because by Definition $1.8$ of \cite{vaku}, proper weights on $S$ extends to proper weights on $\mathcal{M}(S)$. Similarly, a proper weight $\psi$ on $S$ is said to be right invariant if 
\[\psi((\id \otimes \omega(\Delta(a)))=\omega(1)\psi(a)\]
for all positive $a$ such that $\psi(a)<+\infty$ and $\omega\in S^*_+$. We will use the standard notation $[X]$ for the closed linear span of a subset $X$ of a Banach space. Let us recall the definition of a locally compact quantum group in the reduced setting:

\begin{deef}[Definition $4.1$ of \cite{vaku}]
If $S$ satisfies
\[S=[(\id \otimes \omega(\Delta(a)): a\in S, \omega\in S^*]=[(\omega \otimes \id(\Delta(a)): a\in S, \omega\in S^*]\]
and there exist faithful approximate KMS-weights $\phi$ and $\psi$, which should be left invariant respectively right invariant, we say that $S$ is a reduced $C^*$-algebraic quantum group. 
\end{deef}

The weights $\phi$ and $\psi$ are called left, respectively right, Haar weights. By Corollary $6.11$ of \cite{vaku} the density conditions in the definition of a reduced locally compact quantum group together with the existence of Haar weights imply the cancellation condition
\[S\otimes S=[\Delta(S)(1\otimes S)]=[\Delta(S)(S\otimes 1)].\]
See also the von Neumann-algebraic version of a quantum group in \cite{vakutv}. In the von Neumann-setting a density requirement on the quantum group is unnecessary. In \cite{vakutv} the conceptually remarkable result that there exists a unique reduced $C^*$-algebraic quantum group contained in every von Neumann-algebraic quantum group is proved. 

We recall the standard notation $\mathcal{N}_\phi:=\{a\in S: \phi(a^*a)<+\infty\}$. The sesquilinear form $\langle a,b\rangle _\phi:=\phi(a^*b)$ defines a non-degenerate scalar product on $\mathcal{N}_\phi$  since $\phi$ is faithful. The Hilbert space closure of $\mathcal{N}_\phi$ will be denoted by $H_S$. The Hilbert space $H_S$ carries a faithful representation $\lambda:S\to \Bo(H_S)$ given by the GNS-representation associated with $\phi$. The representation $\lambda$ is called the left regular representation of $S$. We denote the embedding $\mathcal{N}_\phi\hookrightarrow H_S$ by $\Lambda_\phi$. Define a linear mapping $W$ on $H_S\otimes H_S$ in the dense subspace $\im (\Lambda_\phi \otimes \Lambda_\phi)$ by 
\[W^*(\Lambda_\phi (a)\otimes \Lambda _\phi(b)):=\Lambda_\phi \otimes \Lambda_\phi(\Delta(b)a\otimes 1) \quad \mbox{for} \quad a,b\in \mathcal{N}_\phi.\]
By Proposition $3.17$ of \cite{vaku} the operator $W$ is a unitary operator such that
\[\lambda\otimes \lambda(\Delta (a))= W^*(1\otimes \lambda(a))W.\]
Furthermore, equation $(4.2)$ of \cite{vaku} states
\[\lambda(S)=[(\id \otimes \omega)(W): \omega \in \Bo (H_S)_*]=[(\id \otimes \omega)(W^*): \omega \in \Bo (H_S)_*].\]

The unitary operator $W$ satisfies the pentagonal equation 
\[W_{12}W_{13}W_{23}=W_{23}W_{12},\]
so we say that $W$ is a multiplicative unitary. We may define an antipode on a dense subalgebra of $S$ using the multiplicative unitary $W$. The antipode is defined as: 
\[\mathcal{S}\left((\id\otimes \omega)(W)\right):=(\id\otimes \omega)(W^*) \quad \mbox{for}\quad \omega\in \Bo(H_S)_*.\]
The mapping $\mathcal{S}$ is well defined and an anti-automorphism of its domain. By Proposition $5.22$ and $5.24$ of \cite{vaku} there is a polar decomposition 
\[\mathcal{S}=R\tau_{-\frac{i}{2}},\]
where $R$ is an anti-automorphism of $S$ and $\tau_{-\frac{i}{2}}$ is the densely defined extension of the strongly continuous one-parameter group $(\tau_t)_{t\in \R}$ associated with the left-invariant weight $\phi$ to $-\frac{i}{2}$. As is shown in Proposition $5.26$ of \cite{vaku} the anti-automorphism $R$ satisfies the equation
\[\sigma\circ (R\otimes R)\circ \Delta=\Delta \circ R,\] 
where $\sigma$ denotes the flip mapping. So the weight $a\mapsto \phi(R(a))$ satisfies the conditions on a right invariant Haar weight and we may assume $\psi\equiv \phi\circ R$. Let $\rho$ be the GNS-representation of $S$ associated with this particular choice of right Haar weight. As is shown in \cite{vaku}, the GNS-construction $\rho$ may be expressed via the antipode $R$, the left regular representation $\lambda$ and the modular conjugation operator $J$ of $\phi$ via
\[\rho(a)=J\lambda(R(a)^*)J.\]
Therefore, we have the inclusions $\rho(S)\subseteq \lambda(S)'$ and $\lambda(S)\subseteq \rho(S)'$ and the representations $\lambda$ and $\rho$ commute. 

Let $\hat{S}$ denote the closed subset in $\Bo(H_S)$ defined as
\begin{equation}
\label{dualdense}
[(\omega \otimes\id )(W): \omega \in \Bo(H_S)_*].
\end{equation}
The space $\hat{S}$ forms a $C^*$-algebra by Proposition $1.4$ and $3.5$ from \cite{baskauni}. Another equivalent approach to the dual $\hat{S}$ is by using the multiplicative unitary $\hat{W}:=\sigma W^*\sigma$ and defining $\hat{S}$ as
\begin{equation}
\label{dualdef}
\hat{S}:=[(\id\otimes \omega)(\hat{W}): \omega \in \Bo(H_S)_*].
\end{equation}
Define the mapping $\hat{\Delta}:\hat{S}\to \Bo(H_S\otimes H_S)$ as
\[\hat{\Delta}(x):= \hat{W}^*(1\otimes x)\hat{W}.\]
By Theorem $8.20$ and $8.29$ of \cite{vaku} there exist Haar weights $\hat{\phi}$ and $\hat{\psi}$ making the pair $(\hat{S},\hat{\Delta})$ into a reduced $C^*$-algebraic quantum group and there exist a natural Pontryagin duality $\hat{\hat{S}}\cong S$.

Similarly to the setting for $S$ we have a left regular and a right regular representation of $\hat{S}$ on $H_S$. We will denote them by $\hat{\lambda}$ and $\hat{\rho}$. Using the definition in equation \eqref{dualdef} we have that $\hat{\lambda}$ coincides with the inclusion $\hat{S}\subseteq \Bo(H_S)$, because of Proposition $8.16$ of \cite{vaku}. The relations between the right and the left regular representations can be described via the modular conjugation operators of its dual as follows
\[\rho(a)=J\hat{J}\lambda(a)\hat{J}J \quad \mbox{and} \quad \hat{\rho}(a)=\hat{J}J\hat{\lambda}(a)J\hat{J}.\]

The unitary $W$ is called the left regular corepresentation of $S$ and $\hat{W}$ is called the left regular corepresentation of $\hat{S}$. From the reasonings in chapter $3.4$ of \cite{vaku} it follows that $W\in (\lambda\otimes \hat{\lambda})(\mathcal{M}(S\otimes \hat{S}))$ and $\hat{W}\in (\hat{\lambda}\otimes \lambda)(\mathcal{M}(\hat{S}\otimes S))$.

Similarly one can define the right regular corepresentation of the quantum group $S$ as the unitary $V\in (\hat{\rho}\otimes \rho)(\mathcal{M}(\hat{S}\otimes S))$ such that 
\[(\rho\otimes \rho)(\Delta(a))=V(\rho(a)\otimes 1)V^*.\]
A straight forward calculation shows 
\[V=(J\hat{J}\otimes J\hat{J})\hat{W}(\hat{J}J\otimes \hat{J} J).\]
The right regular corepresentation $\hat{V}$ of $\hat{S}$ may be defined in the same way as $V$ but for $\hat{S}$. By either Pontryagin duality or another straight forward calculation it follows that 
\[\hat{V}=(\hat{J}J\otimes \hat{J}J)W(J\hat{J}\otimes J\hat{J}).\]
\[\]

Just as in \cite{voigtfree}, if $A$ is a $C^*$-algebra we define 
\[\mathcal{M}_S(A\otimes S):=\{T\in \mathcal{M}(A\otimes S):(1\otimes s)T,\;T(1\otimes s)\in A\otimes S \;\forall s\in S\}.\]
Recall that a non-degenerate $*$-homomorphism $\Delta_A:A\to \mathcal{M}_S(A\otimes S)$ is called a coaction if it is coassociative in the sense that $(\id \otimes \Delta)\Delta_A=(\Delta_A\otimes \id)\Delta_A$. If $\Delta_A$ is faithful, we say that the coaction $\Delta_A$ is reduced. If $\Delta_A(A)\cdot 1\otimes S$ is dense in $A\otimes S$ we say that the coaction is continuous. We may associate a representation $\lambda_A:A\to \ellL(A\otimes H_S)$ with a coaction by $\lambda_A:=(\id\otimes \lambda) \circ \Delta_A$. The reduced  crossed product $A\rtimes_r  S$ is defined as 
\[A\rtimes_r S:=[\lambda_A(A)\cdot 1\otimes \hat{\rho}(\hat{S})]\subseteq \ellL(A\otimes H_S).\]

\begin{prop}
\label{contcro}
The space $A\rtimes_r  S$ forms a $C^*$-algebra with a continuous, reduced coaction of $\hat{S}$ defined by 
\[\Delta_{A\rtimes_r  S}(a\rtimes  \hat{s}):=\lambda_A(a)\cdot 1\otimes (\hat{\rho}\otimes \id)\hat{\Delta}(\hat{s}).\]
\end{prop}

The coaction $\Delta_{A\rtimes_r S}$ is called the dual coaction on $A\rtimes_r  S$ and was defined in Definition $7.3$ of \cite{baskauni}. 

\begin{proof}
As is shown in Lemma $7.2$ in \cite{baskauni}, the closed linear span of $\lambda_A(A)\cdot 1\otimes \hat{\rho}(\hat{S})$ is a $*$-algebra. So $A\rtimes_r  S$ forms a $C^*$-algebra. Let $a\rtimes  \hat{s}$ denote the element $\lambda_A(a)1\otimes\hat{\rho}( \hat{s})$. The dual coaction on $A\rtimes_r S$ is reduced since it is implemented by the unitary $V$ as 
\[(\id_{A\rtimes_rS}\otimes \hat{\rho})\Delta_{A\rtimes_r  S}(a\rtimes  \hat{s})=Ad(V_{23})\left(\lambda_A(a)_{12}\cdot1\otimes\hat{\rho}( \hat{s})\otimes 1\right).\]
Since $\hat{S}$ is a reduced locally compact quantum group, $\hat{\Delta}(\hat{S})\cdot 1\otimes \hat{S}$ is dense in $\hat{S}\otimes \hat{S}$ and 
\begin{align*}
[\Delta_{A\rtimes_r S}(A\rtimes _r S)\cdot 1_{A\rtimes_r S}\otimes \hat{S}]&=[\lambda_A(A)_{12}\cdot (\hat{\rho}\otimes \id)\hat{\Delta}(\hat{S}))_{23}\cdot 1_{A\rtimes_rS}\otimes \hat{S}]=\\
&=[\lambda_A(A)_{12}\cdot 1\otimes \hat{\rho}(\hat{S})\otimes \hat{S}]=(A\rtimes_r S)\otimes \hat{S}
\end{align*}
which proves that the dual coaction is continuous.
\end{proof}

\section{Twists of reduced quantum groups}

Twists plays a central role in quantization of Lie algebras. The twists can be seen as quantized infinitesimal group cocycles of the quantized Lie group. The main examples of locally compact quantum groups have been compact quantum groups. The usual source of examples of compact quantum groups comes from deformations. Deformations are difficult to fit into a general framework since the underlying $C^*$-algebra and its deformation are very hard to relate because they have, by construction, different multiplication. Since deformations are dual to twists, one can choose the equivalent approach of using twists in the study of deformations.

\begin{deef}
\label{twistdef}
If the unitary $\Fg \in \mathcal{M}(S\otimes S)$ satisfies the density condition:
\[[\omega\otimes \id(\Fg\Delta(a)\Fg^*):\omega \in S^*]=[ \id\otimes \omega(\Fg\Delta(a)\Fg^*):\omega\in S^*]= S,\] 
and the cocycle condition 
\[Ad\left(\Fg_{23}(\id\otimes \Delta)(\Fg)\right)(\Delta^{(2)}(s))=Ad\left(\Fg_{12}(\Delta \otimes \id )(\Fg)\right)(\Delta^{(2)}(s)) \quad \forall s\in S,\]
$\Fg$ is called a cocycle twist of $S$. If a cocycle twist $\Fg$ satisfies the stronger condition 
\[\Fg_{23}(\id\otimes \Delta)(\Fg)=\Fg_{12}(\Delta \otimes \id )(\Fg),\] 
we say that $\Fg$ is a twist. 
\end{deef}

Because of the cocycle condition we may define a new, coassociative comultiplication $\Delta_\Fg:=Ad \Fg \; \circ\Delta$ on $S$. The cocycle condition is in fact equivalent to $\Delta_\Fg$ being coassociative. If $\Fg$ is a cocycle twist we let $S_\Fg$ denote the bi-$C^*$-algebra $S$ with comultiplication $\Delta_\Fg$. We define the associator of $\Fg$ as:
\[\Phi_\Fg:=(\id\otimes \Delta)(\Fg^*)\Fg_{23}^*\Fg_{12}(\Delta \otimes \id )(\Fg).\]
The cocycle condition for $\Fg$ is equivalent to $\Phi_\Fg$ commuting with $\Delta^{(2)}(S)$. 

Observe that if $\Fg$ is a cocycle twist of $S$ and $U\in \mathcal{M}(S)$ we can define a new twist by $\Fg':=\Delta_\Fg(U)\Fg$. The unitary $U$ induces an isomorphism between the two bi-$C^*$-algebras $S_\Fg$ and $S_{\Fg'}$ since $\Delta_\Fg(U^*sU)=\Delta_{\Fg'}(s)$.\\

To explain the cocycle condition for $\Fg$ we can take $S$ as a quantum group admitting a counit and let $\mathfrak{C}$ denote the representation category of $S$. The objects of $\mathfrak{C}$ are representations $\pi:S\to \Bo(H)$ and a morphism $\pi_1\to\pi_2$ is an intertwining operator. The category $\mathfrak{C}$ becomes a monoidal category by defining $\pi_1\otimes_\mathfrak{C} \pi_2(s):=(\pi_1\otimes \pi_2)\circ \Delta(s)$ and the identity object is given by the counit $\epsilon:S\to \C$. Since $\Delta$ is coassociative, $\mathfrak{C}$ is a strict monoidal category.

We can normalize $\Fg$ such that $\id\otimes \epsilon \otimes \id(\Phi_\Fg)=1$. When $\Fg$ is normalized we may twist the tensor product $\otimes _\mathfrak{C}$ by $\Fg$ as  $\pi_1\otimes_{\mathfrak{C}_\Fg} \pi_2(s):=(\pi_1\otimes \pi_2)\circ \Delta_\Fg(s)$. Letting $\mathfrak{C}_\Fg$ denote the monoidal category given by equipping $\mathfrak{C}$ with the tensor product $\otimes _{\mathfrak{C}_\Fg}$ and let the associator of $\mathfrak{C}_\Fg$ be $\Phi_\Fg$. The representation category $\mathfrak{C}_\Fg$ is strict if and only if $\Fg$ is a twist.

\begin{deef}
\label{feasi}
If the bi-$C^*$-algebra $S_\Fg$ admits the structure of a reduced locally compact quantum group, we say that the cocycle twist $\Fg$ is feasible.
\end{deef}

To study twists we recall the following theorem from \cite{deco}:

\begin{sats}[Corollary $6.2$ and Proposition $6.4$ of \cite{deco}]
\label{decothm}
Let $M$ denote the von Neumann algebra generated by $S$ in its left regular representation. If $\Fg\in M\overline{\otimes} M$ satisfies 
\[\Fg_{23}(\id\otimes \Delta)(\Fg)=\Fg_{12}(\Delta \otimes \id )(\Fg).\]
there exist Haar weights with respect to $\Delta_\Fg$ on $M$ and the left regular corepresentation of $S_\Fg$ is given by 
\[W_\Fg= (\hat{J}_\Fg\otimes J)\Fg W^*(\hat{J}\otimes J)\Fg^*\]
where $\hat{J}_\Fg$ denote the modular conjugation operator of $\widehat{M_\Fg}$.
\end{sats}

For the proof of Theorem \ref{decothm} we refer the reader to \cite{deco}. The proof is rather technical in nature and consists of using the twist $\Fg$ to define a Galois object over $M$. The content of Theorem \ref{decothm} together with Proposition $1.6$ of \cite{vakutv} implies that if $\Fg$ is a twist of $S$, there is a unique $C^*$-subalgebra $S(\Fg)\subseteq M$ of the von Neumann closure of $S$ that becomes a reduced locally compact quantum group in the comultiplication $\Delta_\Fg$. Clearly, $\Fg$ is feasible if and only if $S(\Fg)=S_\Fg$. However, it is quite hard to determine $S(\Fg)$ in general. In \cite{deco} a twist of a compact quantum group was constructed in such a way that the twisted quantum group was not compact. We will in this paper mainly focus on feasible twists.

\begin{prop}
If $\Fg$ is a (feasible) twist of $S$, $\Fg^*$ is a (feasible) twist of $S(\Fg)$ and $S(\Fg)(\Fg^*)=S$.
\end{prop}

\begin{proof}
If $\Fg$ is a twist of $S$ the strong cocycle condition for $\Fg$ implies that
\[\Fg^*_{23}(\id\otimes \Delta_\Fg)(\Fg^*)=(\id\otimes \Delta)(\Fg^*)\Fg^*_{23}=\]
\[=(\Delta\otimes\id)(\Fg^*)\Fg^*_{12}=\Fg^*_{12}(\Delta_\Fg\otimes\id)(\Fg^*).\]
Thus, if $\Fg$ is a twist of $S$ the unitary $\Fg^*$ satisfies the strong cocycle condition with respect to $\Delta_\Fg$ and is a twist of $S(\Fg)$. Clearly $(\Delta_\Fg)_{\Fg*}=\Delta$ as $*$-homomorphisms on the von Neumann algebra $M$, so the uniqueness of Haar weights, see \cite{vaku}, implies that $S(\Fg)(\Fg^*)=S$.
\end{proof}

Assume that $\Fg$ is a feasible twist of $S$. From $\Fg$ we can define a twisted coaction of the quantum group $S_\Fg$ on $C^*$-algebras of the form $A=B\rtimes_r \hat{S}$. Let $U_{\Fg}^{\rho}$ denote the intertwining unitary between the right regular representation $\rho$ of $S$ and the right regular representation $\rho_\Fg$ of $S_\Fg$, so 
\[\rho_\Fg=Ad(U_{\Fg}^{\rho})\circ\rho.\]
Define the $C^*$-algebra
\[A_\Fg:=Ad(1_B\otimes U_{\Fg}^{\hat{\rho}})(B\rtimes_r \hat{S}).\]
We will denote the elements of $A_\Fg$ by $b\rtimes_\Fg s:=Ad(1\otimes  U_{\Fg}^{\rho})(b\rtimes s)$. Define the morphism $\lambda_B^\Fg:B\to \mathcal{M}(B\otimes \Ko(H_{S_\Fg}))$ by
\[\lambda_B^\Fg:=Ad(1_A\otimes U_{\Fg}^{\rho})\circ \lambda_B.\]
We define the $*$-homomorphism
\[\Delta_{A_\Fg}:A_\Fg\to \mathcal{M}_{S_\Fg}(A_\Fg\otimes S_\Fg ),\]  
by defining it in the right regular representation $\rho_\Fg$ of $S_\Fg$ using the right regular corepresentation $V_{\Fg}$ 
\begin{equation}
\label{twistcoac}
(\id_{A_\Fg}\otimes \rho_\Fg)\Delta_{A_\Fg}(b\rtimes_\Fg s):=Ad(V_{\Fg,23})\left(b\rtimes_\Fg s\otimes 1\right)=
\end{equation}
\[=Ad(V_{\Fg,23})\left(\lambda_B^\Fg(b)_{12}\cdot 1\otimes \rho_\Fg(s)\otimes 1\right).\]

\begin{prop}
\label{contcoatw}
The $*$-homomorphism $\Delta_{A_\Fg}$ is a well defined coaction of $S_\Fg$ on $A_\Fg$ which is continuous and reduced.
\end{prop}

The coaction $\Delta_{A_\Fg}$ will be called the twisted coaction on $A_\Fg$ by $\Fg$. 

\begin{proof}
From the pentagonal identity for $V_\Fg$ it follows that $\Delta_{A_\Fg}$ is coassociative. To prove the inclusion $\Delta_{A_\Fg}(A_\Fg)\subseteq \mathcal{M}_{S_\Fg}(A_\Fg\otimes S_\Fg)$, we observe that if we take $a\in A$ and $s_1,s_2\in S$
\[(\id_{A_\Fg}\otimes \rho_\Fg)\Delta_{A_\Fg}(b\rtimes_\Fg  s_1)\cdot 1_{A_\Fg}\otimes \rho_\Fg(s_2)=\]
\[=V_{\Fg,23}\lambda_B^\Fg(b)_{12}V_{\Fg,23}^*\cdot \left((\rho_\Fg\otimes \rho_\Fg(\Delta_\Fg(s_1)))\cdot 1_{S_\Fg}\otimes\rho_\Fg(s_2)\right)_{23}.\]
Since $\Delta_\Fg(s_1) \cdot 1_{S}\otimes s_2\in S_\Fg\otimes S_\Fg$ it follows that $\Delta_{A_\Fg}(b\rtimes_\Fg  s_1)\cdot 1_{A_\Fg}\otimes s_2\in  \mathcal{M}_{S_\Fg}(A_\Fg\otimes S_\Fg)$.

The $S_\Fg$-coaction on $A_\Fg$ is reduced since it is by definition implemented by the unitary $V_\Fg$. From the definition of $\Delta_{A_\Fg}$ we have 
\[[(\id_{A_\Fg}\otimes \rho_\Fg)\left(\Delta_{A_\Fg}(A_\Fg)\right)\cdot 1_{A_\Fg}\otimes\rho_\Fg(S)]=\]
\[=[V_{\Fg,23} \cdot\lambda_B^\Fg(B)_{12}\cdot 1_B\otimes\rho_\Fg(S)\otimes 1_{H_S}\cdot V_{\Fg,23} ^*\cdot 1_{A_\Fg}\otimes\rho_\Fg(S)].\]
If we consider the closed operator space $[(\rho_\Fg(S)\otimes 1_{H_S}) V_{\Fg} ^*( 1_{H_S}\otimes\rho_\Fg(S))]$ the fact that $\Delta_\Fg(S)\cdot S\otimes 1$ is dense in $S\otimes S$ implies that 
\[[(\rho_\Fg(S)\otimes 1_{H_S}) V_{\Fg} ^*( 1_{H_S}\otimes \rho_\Fg(S))]=[V_\Fg^*(\rho_\Fg(S)\otimes\rho_\Fg(S))],\] 
and
\begin{align*}
&[(\id_{A_\Fg}\otimes \rho_\Fg)\left(\Delta_{A_\Fg}(A_\Fg)\right)\cdot 1_{A_\Fg}\otimes\rho_\Fg(S)]=\\
=&\,[V_{\Fg,23}\lambda_B^\Fg(B)_{12}V_{\Fg,23}^*\cdot 1_B\otimes\rho_\Fg(S)\otimes \rho_\Fg(S)].
\end{align*}
Therefore 
\begin{align*}
&V_{\Fg,23}^*[(\id_{A_\Fg}\otimes\rho_\Fg)\left(\Delta_{A_\Fg}(A_\Fg)\right)\cdot 1_{A_\Fg}\otimes\rho_\Fg(S)]V_{\Fg,23}=\\
=&\,V_{\Fg,23}^*[V_{\Fg,23}\lambda_B^\Fg(B)_{12}V_{\Fg,23}^*\cdot 1_B\otimes\rho_\Fg(S)\otimes \rho_\Fg(S)]V_{\Fg,23}=\\
=&\,[\lambda_B^\Fg(B)_{12}\cdot 1_B\otimes\rho_\Fg(S)\otimes \rho_\Fg(S)]=A_\Fg\otimes \rho_\Fg(S).
\end{align*}
Thus $[\Delta_{A_\Fg}(A_\Fg)\cdot 1_{A_\Fg}\otimes\rho_\Fg(S)]=A_\Fg\otimes\rho_\Fg(S)$ since  
\[[\Delta_{A_\Fg}(A_\Fg)\cdot 1_{A_\Fg}\otimes S]\subseteq A_\Fg \otimes S.\]
\end{proof}

Let us recall the definition of a regular quantum group from \cite{baskauni}. Regularity may be defined in terms of $S$ left regular corepresentation $W$. We let $\sigma$ denote the flip map on $H_S\otimes H_S$. Following Proposition $3.2$ of \cite{baskauni} we define the algebra
\[\mathcal{C}(W):=[\{\id\otimes \omega(\sigma W): \omega\in \Bo(H_S)_*\}].\]
If $\mathcal{C}(W)=\Ko(H_S)$ the locally compact quantum group $S$ is said to be regular.

\begin{deef}
\label{strongfeasi}
If $S(\Fg)$ is regular we say that $\Fg$ is a regular twist.
\end{deef}

Since every discrete quantum group is regular and admits Haar weights we have the following proposition:

\begin{prop}
\label{codireg}
If $S$ is discrete, every cocycle twist of $S$ is regular and feasible.
\end{prop}

Assuming regularity of a twist allows us to obtain a twisted Takesaki-Takai duality for regular quantum groups. The untwisted Takesaki-Takai duality was obtained in Theorem $7.5$ of \cite{baskauni}. It states that there is an equivariant isomorphism $A\rtimes_r  S\rtimes_r  \hat{S}\cong\Ko(H_S)\otimes A$. The Takesaki-Takai isomorphism does, in fact, behave well even if we allow a regular twist.

\begin{sats}[TTT-duality]
\label{tata}
Let $S$ be a regular, reduced locally compact quantum group and $\Fg$ a regular, feasible twist of $S$. If $B$ has a reduced, continuous $\hat{S}$-coaction and $A=B\rtimes_r \hat{S}$ there is a natural $*$-isomorphism 
\[A_\Fg \rtimes_r S_\Fg\cong A\rtimes_r S.\] 
\end{sats}

Before we prove Theorem \ref{tata}, we need some notations. We will denote the dual quantum group of $S_\Fg$ by $\hat{S}_\Fg$ and the left regular representation of $\hat{S}_\Fg$ will be denoted by $\hat{\lambda}_\Fg$ and similarly the right regular representation of $\hat{S}_\Fg$ will be denoted by $\hat{\rho}_\Fg$. Let $W_\Fg$ denote the left regular corepresentation of $S_\Fg$ and denote the modular conjugation operator of the left Haar weight $\hat{\phi}_\Fg$ of $\hat{S}_{\Fg}$ by $\hat{J}_\Fg$.

\begin{proof}
Let us start by observing that 
\[A_\Fg \rtimes_r S_\Fg=[\tilde{\Fg}_{23}\lambda_B^\Fg(B)_{12}\tilde{\Fg}_{23}^*\cdot (1\otimes (\rho_\Fg\otimes\lambda_\Fg)\Delta_\Fg(S))\cdot (1\otimes 1\otimes \hat{\rho}_\Fg(\hat{S}_\Fg))],\]
where $\tilde{\Fg}:=\rho_\Fg\otimes\lambda_\Fg(\Fg)$.

Consider the $*$-homomorphism 
\[\kappa:=Ad(1\otimes ( \hat{J}_\Fg\otimes \hat{J}_\Fg) W_\Fg( \hat{J}_\Fg\otimes \hat{J}_\Fg)):A_\Fg\rtimes_r S_\Fg\to \mathcal{M}_{\Ko\otimes \Ko}(B\otimes \Ko\otimes \Ko),\]
whose image coincide with 
\[\im\kappa=[\lambda_B'(B)\cdot (1_{B\otimes H_S}\otimes \hat{\lambda}_\Fg(\hat{S}_\Fg)\lambda_\Fg(S))],\]
where 
\[\lambda'_B(b):=Ad(1\otimes (\hat{J}_\Fg\otimes \hat{J}_\Fg) W_\Fg(\hat{J}_\Fg\otimes \hat{J}_\Fg)\tilde{\Fg})(\lambda_B^\Fg(b)).\] 
However, Takesaki-Takai duality implies that 
\[\im \kappa=[\lambda_B'(B)\cdot 1_{B\otimes H_S}\otimes \Ko(H_{S_\Fg})]\cong B\otimes \Ko\cong B\rtimes_r \hat{S}\rtimes_r S=A\rtimes_r S.\]
\end{proof}

\section{Torsion-free discrete quantum groups}

In the paper \cite{mehom}, Meyer introduced the notion of torsion-free discrete quantum groups. When $\Gamma$ is a discrete group, $C_0(\Gamma)$ is torsion-free in the sense of Meyer if and only if $\Gamma$ is torsion-free. For compact groups $G$ the dual quantum group $\hat{G}$, or actually the bi-$C^*$-algebra $C^*(G)$, is torsion-free precisely when $G$ satisfies the Hodgkin condition, that is $G$ is connected with torsion-free fundamental group. As is shown in \cite{rosch}, there exists Künneth formulas and UCT:s in $KK^G$ for compact Lie groups $G$ satisfying the Hodgkin condition. So the notion of torsion-free discrete quantum groups simplifies much of the homological algebra for $C^*$-algebras. 

\begin{deef}[\cite{mehom}]
A discrete quantum group $S$ is torsion-free if every coaction of $\hat{S}$ on a finite-dimensional $C^*$-algebra $A$ is equivariantly Morita equivalent to direct sums of $\C$ with the trivial coaction.
\end{deef}

Suppose that $G$ is a compact, connected Lie group satisfying the Hodgkin condition. Let $T\subseteq G$ be a maximal torus of rank $n$ and $w$ the Weyl group. By \cite{stei}, the representation ring $R(T)$ is a free $R(G)$-module of rank $|w|$ if $G$ satisfies the Hodgkin condition. As is shown in \cite{nemetva},  the fact that $R(T)$ is a free $R(G)$-module implies that there exists a natural isomorphism $KK^G(A,\C^{|w|})\cong KK^T(A,\C)$. So the fact that the quantum group $\hat{G}$ is torsion-free follows from the classical result that $\hat{T}=\Z^n$ is torsion-free. \\

\begin{lem}
\label{twotimmor}
Suppose that $\Fg$ is a regular, feasible twist of the regular, reduced locally compact quantum group $S$. If $B$ has a continuous, reduced coaction of $\hat{S}$ there exist a $C^*$-algebra $B^\Fg$ with a continuous, reduced coaction of $\hat{S}_{\Fg}$ such that $B\cong B^\Fg$ as $C^*$-algebras and an equivariant Morita equivalence 
\[(B\rtimes_r \hat{S})_\Fg\rtimes_r S_\Fg\sim _M B^\Fg.\]
\end{lem}

\begin{proof}
Letting $B^\Fg:=B$ as a $C^*$-algebra we know that $(B\rtimes_r \hat{S})_\Fg\rtimes_r S_\Fg\cong B^\Fg\otimes \Ko(H_{S_\Fg})$ by TTT-duality. If $B^\Fg$ may be given a coaction of $\hat{S}_{\Fg}$ such that the coaction on $(B\rtimes \hat{S})_\Fg\rtimes S_\Fg$ coincides with that of $B^\Fg\rtimes_r \hat{S}_\Fg\rtimes_r S_\Fg$, the statement of the lemma follows. Let $B_s^\Fg:=B^\Fg\otimes \Ko(H_{S_\Fg})$ with the coaction of $\hat{S}_{\Fg}$ induced from the TTT-isomorphism $B_s^\Fg\cong (B\rtimes_r \hat{S})_\Fg\rtimes_r S_\Fg$ from Theorem \ref{tata}. Letting $V_\Fg$ denote the right regular corepresentation of $S_{\Fg}$, the coaction on $B_s^\Fg$ is induced by $V_\Fg$. 

On the other hand, suppose that we have a continuous, reduced coaction $\Delta_{B^\Fg}$ of $\hat{S}_\Fg$ on $B^\Fg$. If we define $\tilde{B}_s^\Fg:=B^\Fg\otimes \Ko(H_{S_\Fg})$ with coaction induced from the Takesaki-Takai isomorphism $\tilde{B}_s^\Fg\cong (B^\Fg\rtimes_r \hat{S}_\Fg)\rtimes_r S_\Fg$ the coaction on $b\otimes k\in \tilde{B}_s^\Fg$ would be given by
\[\Delta_{\tilde{B}_s^\Fg}(b\otimes k)=
V_{\Fg,23} (\Delta_{B^\Fg} (b)_{13}\cdot 1_{B}\otimes k\otimes 1) V_{\Fg,23}^*,\]
see more in Theorem $7.5$ of \cite{baskauni}. Using this reasoning as motivation, we define the $*$-homomorphism $\delta_\Fg:B^\Fg\to \mathcal{M}(B_s^\Fg\otimes \hat{S}_\Fg)$ by 
\[\delta_\Fg (b)=V_{\Fg,23}^*\Delta_{B_s^\Fg}(b\otimes 1)V_{\Fg,23}.\]
The morphism $\delta_\Fg$ restricts to a coaction $\Delta_{B^\Fg}:B^\Fg\to \mathcal{M}_{\hat{S}_\Fg}(B^\Fg\otimes \hat{S}_\Fg)$ since for all $b\in B$, $k\in \Ko(H_{S_\Fg})$ we have the equality
\begin{align*}
\delta_\Fg (b)\cdot 1\otimes k\otimes 1 &=V_{\Fg,23}^*\Delta_{B_s^\Fg}(b\otimes 1)V_{\Fg,23}\cdot 1\otimes k\otimes 1=\\
=V_{\Fg,23}^*\Delta_{B_s^\Fg}(b\otimes k)V_{\Fg,23}&=1\otimes k\otimes 1\cdot V_{\Fg,23}^*\Delta_{B_s^\Fg}(b\otimes 1)V_{\Fg,23}=\\
&=1\otimes k\otimes 1\cdot\delta_\Fg (b).
\end{align*}
So $\delta_\Fg (b)\in (1\otimes \Ko(H_{S_\Fg})\otimes 1)'$, therefore it is of the form $\delta_\Fg (b)=\Delta_{B^\Fg} (b)_{13}$ for a unique element $\Delta_{B^\Fg} (b)\in \mathcal{M}(B\otimes \hat{S}_\Fg)$. Since $\Delta_{B_s^\Fg}$ is a reduced coaction $\Delta_{B^\Fg} (b)\in\mathcal{M}_{\hat{S}_\Fg}(B^\Fg\otimes \hat{S}_\Fg)$ and $\Delta_{B^\Fg}$ defines a reduced coaction on $B^\Fg$. The coaction $\Delta_{B^\Fg}$ is continuous since
\[[\Delta_{B^\Fg}(B^\Fg)_{13}\cdot 1\otimes \Ko\otimes 1\cdot 1\otimes 1\otimes \hat{S}_\Fg]=[\Delta_{B_s^\Fg}(B_s^\Fg)\cdot 1_{B_s^\Fg}\otimes \hat{S}_\Fg]=B\otimes \Ko\otimes \hat{S}_\Fg.\]

Let $B^\Fg$ be given the coaction $\Delta_{B^\Fg}$, it follows directly from the definition and the arguments above that $\Delta_{B_s^\Fg}=\Delta_{\tilde{B}_s^\Fg}$. Thus we have the following sequence of equivariant Morita equivalences
\[(B\rtimes_r \hat{S})_\Fg\rtimes_r S_\Fg\cong (B^\Fg\rtimes_r \hat{S}_\Fg)\rtimes_r S_\Fg \sim _M B.\]
\end{proof}

\begin{sats}
\label{twistdisc}
If $S$ is discrete, torsion-free and $\Fg$ a twist of $S$, $S_\Fg$ is also discrete, torsion-free. 
\end{sats}

\begin{proof}
The property of being discrete is invariant under twist. Suppose that $A\in C^*_{\hat{S}_\Fg}$ is finite-dimensional. Consider the object $\tilde{A}:=(A\rtimes_r \hat{S}_\Fg)_{\Fg^*}\in C^*_S$. Since $A$ is of finite dimension, Lemma \ref{twotimmor} implies that there is an equivariant Morita equivalence from $\tilde{A}\rtimes_r S$ to a finite-dimensional $C^*$-algebra. Since $S$ is torsion-free,  $\tilde{A}\rtimes_r S$ is Morita equivalent to $\C^k$ for some $k$. Therefore Takesaki-Takai duality implies that $\tilde{A}_\Fg\sim _M (\C^k\rtimes \hat{S})_\Fg$. Lemma \ref{twotimmor} implies that $\tilde{A}_\Fg\rtimes_r S_\Fg\sim \C^k$ since any action on $\C$ is trivial. Collecting all these results, we obtain that
\[A\sim_M ((A\rtimes_r \hat{S}_\Fg)_{\Fg^*})_\Fg\rtimes_r S_\Fg\sim_M\tilde{A}_\Fg\rtimes _r S_\Fg\sim_M\C^k.\]
\end{proof}

In \cite{mehom} the question was posed whether duals of the Woronowicz deformations are torsion-free? Theorem \ref{twistdisc} does unfortunately not answer the question since the Drinfeld-Jimbo twists are not twists. The torsion of the duals of the Woronowicz deformations must be dealt with care, the proof that duals of the Woronowicz deformations are torsion-free uses the generators of the Drinfeld-Jimbo algebra explicitly. This indicates that the torsion a general cocycle twist may produce might require some heavier machinery to be dealt with.

\section{Drinfeld-Jimbo twists}
The aim of this section is to give the structure of a locally compact quantum group to the Drinfeld-Jimbo quantization of a Lie algebra and describe it as a feasible twist of the group $C^*$-algebra of the Lie group. The Drinfeld-Jimbo twists have been studied from the viewpoint of  operator algebras in \cite{netu} and in \cite{rosso} the Drinfeld-Jimbo twists were proven to be dual to the Woronowicz deformations. Our aim is to construct a locally compact quantum group $\hat{G}_q$ such that its finite dimensional modules are weight modules of the Drinfeld-Jimbo quantization $\mathcal{U}_q(\mathfrak{g})$, which are twists of a classical $G$-module. The construction is very much inspired by the notes \cite{netu} and we follow their notations. The Drinfeld-Jimbo algebra $\mathcal{U}_q(\mathfrak{g})$ is constructed for $q>0$ as a deformation of the complex universal enveloping algebra $\mathcal{U}(\mathfrak{g})$ where $\mathfrak{g}$ is the complexified Lie algebra of $G$. To read more about the algebra $\mathcal{U}_q(\mathfrak{g})$ and its representation theory, see \cite{klimsch}. We will recall the definition of $\mathcal{U}_q(\mathfrak{g})$. 

Denote the Cartan matrix of $\mathfrak{g}$ associated with the Cartan subalgebra $\mathfrak{h}$ by $A=(a_{ij})$. We take coprime integers $d_1,\ldots, d_n$ such that the matrix $(d_ia_{ij})$ is symmetric. For a number $q>0$ we define $q_i:=q^{d_i}$. We set $\mathcal{U}_1(\mathfrak{g}):=\mathcal{U}(\mathfrak{g})$ and if $q\neq 1$ the algebra  $\mathcal{U}_q(\mathfrak{g})$ is generated by the symbols $E_i,F_i,K_i,K_i^{-1}$, for $1\leq i\leq n$, satisfying the relations
\begin{equation}
\label{treett}
K_iK_i^{-1}=K_i^{-1}K_i=1, \quad K_iK_j=K_jK_i, \quad K_iE_jK_i^{-1}=q_i^{a_{ij}}E_j, 
\end{equation}
\begin{equation}
\label{tretva}
K_iF_jK_i^{-1}=q_i^{-a_{ij}}F_j,\quad [E_i,F_j] = \delta_{ij}\frac{K_i-K_i^{-1}}{q_i-q_i^{-1}}, 
\end{equation}
\begin{equation}
\label{tretre}
\sum_{k=0}^{1-a_{ij}} (-1)^k
\begin{pmatrix} 1-a_{ij}\\
k \end{pmatrix} _{q_i} E_i^{1-a_{ij}-k}E_jE_i^k=0, 
\end{equation}
\begin{equation}
\label{trefyr}
\sum_{k=0}^{1-a_{ij}} (-1)^k
\begin{pmatrix} 1-a_{ij}\\
k \end{pmatrix} _{q_i} F_i^{1-a_{ij}-k}F_jF_i^k=0.
\end{equation}
The algebra  $\mathcal{U}_q(\mathfrak{g})$ forms a $*$-Hopf algebra with comultiplication
\[\Delta_q(K_i):=K_i\otimes K_i, \quad \Delta_q(E_i):=E_i\otimes 1 +K_i\otimes E_i, \quad \Delta_q(F_i):=F_i\otimes K_i^{-1} + F_i \otimes 1,\]
counit $\epsilon _q:\mathcal{U}_q(\mathfrak{g})\to \C$ defined by $\epsilon_q(E_i)=\epsilon_q(F_i)=0$, $\epsilon _q(K_i)=1$ and the $*$-structure on  $\mathcal{U}_q(\mathfrak{g})$ is given by 
\begin{equation}
\label{trefem}
K_i^*=K_i, \quad E_i^*=K_iF_i \quad \mbox{and} \quad F_i^*=E_iK_i^{-1}.
\end{equation}

To use the notation from \cite{netu}, we let $P\subseteq \mathfrak{h}^*$ denote the set of integral weights. Given a finite dimensional $\mathcal{U}_q(\mathfrak{g})$-module $V$  and a weight $\lambda \in P$ we let $V(\lambda)$ denote the subspace of vectors in $V$ of weight $\lambda$. If $V$ splits as a direct sum of weight modules $V=\oplus_{\lambda\in P}V(\lambda)$ we call $V$ an admissible $\mathcal{U}_q(\mathfrak{g})$-module. The category of admissible $\mathcal{U}_q(\mathfrak{g})$-modules forms a semisimple category with simple generators indexed by dominant integral weights $\lambda \in P_+$. Just as in \cite{netu} we fix  a simple generator $V_\lambda^q$ to each $\lambda\in P_+$.

We let $P_G$ denote the set of weights whose representations $\pi_\lambda$ integrates to a finite-dimensional, unitary representation of $G$ and define the algebra
\[\widehat{\C[G_q]}:=\bigoplus_{\lambda\in P_+\cap P_G} End(V_\lambda^q).\]
This definition differs somewhat from the definition in \cite{netu} which produces the simply connected covering of $\widehat{\C[G_q]}$. The simply connected covering is defined in \cite{netu} as 
\[\widehat{\C[\tilde{G}_q]}:=\bigoplus_{\lambda\in P_+} End(V_\lambda^q).\]

There is a canonical projection $p_G:\widehat{\C[\tilde{G}_q]}\to \widehat{\C[G_q]}$. The multiplier Hopf $*$-algebra structure on $\mathcal{U}_q(\mathfrak{g})$ induces a multiplier Hopf $*$-algebra structure on $\widehat{\C[\tilde{G}_q]}$, since $\mathcal{U}_q(\mathfrak{g})$ is dense in the algebraic multiplier algebra of $\widehat{\C[\tilde{G}_q]}$ equipped with the direct product topology. A finite-dimensional $\widehat{\C[\tilde{G}_q]}$-module correspond to the admissible $\mathcal{U}_q(\mathfrak{g})$-modules, so the algebra $\ker \;p_G$ is a Hopf $*$-ideal and the algebra $\widehat{\C[G_q]}$ also forms a multiplier Hopf $*$-algebra. Let $\Delta_q$ denote the comultiplication of $\widehat{\C[G_q]}$.

The definition of $\widehat{\C[G_q]}$ is motivated by the classical limit in which $\widehat{\C[G]}$ is the convolution algebra of finite dimensional representations of $G$ because of the Peter-Weyl theorem. In the classical limit $q=1$ the multiplier Hopf algebra $\widehat{\C[\tilde{G}]}$ corresponds to the convolution algebra of finite dimensional representations of the simply connected covering group $\tilde{G}$. The $*$-algebra $\widehat{\C[G]}$ is a multiplier Hopf algebra; let $\Delta$ denote its comultiplication. The next theorem is a summary of the results in \cite{netu} relating the two bi-algebras $\widehat{\C[\tilde{G}_q]}$ and $\widehat{\C[\tilde{G}]}$.

\begin{sats}[\cite{netu}]
\label{twistnetu}
For $q>0$ there is a  $*$-isomorphism 
\[\tilde{\phi}:\widehat{\C[\tilde{G}_q]}\to \widehat{\C[\tilde{G}]}\] 
extending the identification of the centers of $\widehat{\C[\tilde{G}_q]}$ and $\widehat{\C[\tilde{G}]}$ and a unitary $\tilde{\Fg}_q \in \prod_{\lambda, \mu\in P_+} End(V_\lambda\otimes V_\mu)$  satisfying the cocycle condition in Definition \ref{twistdef} and 
\[(\tilde{\phi}\otimes\tilde{ \phi}) \circ\tilde{ \Delta}_q=Ad \tilde{\Fg}_q \circ \tilde{\Delta} \circ \tilde{\phi},\]
where $\tilde{\Delta} _q$ denotes the comultiplication of $\widehat{\C[\tilde{G}_q]}$ and $\tilde{\Delta}$ the comultiplication of $\widehat{\C[\tilde{G}]}$. The associator of $\tilde{\Fg}_q$ coincides with the Drinfeld associator $\Phi_{KZ}$.
\end{sats}

The results from \cite{netu} can be applied to an arbitrary semisimple, compact Lie group. Let us first look at the algebraic statement, and then define the $C^*$-algebraic quantum group.

\begin{sats}
\label{djtwtva}
Let $G$ be a semisimple, compact, connected Lie group and $q>0$. There exist a unitary $\Fg_q \in \prod_{\lambda, \mu\in P_+\cap P_G} End(V_\lambda\otimes V_\mu)$ satisfying the cocycle condition in Definition \ref{twistdef} and a $*$-isomorphism $\phi:\widehat{\C[G_q]}\to \widehat{\C[G]}$ such that 
\[(\phi\otimes \phi) \circ \Delta_q=Ad \Fg_q \circ \Delta \circ \phi.\]
\end{sats}

\begin{proof}
The $*$-isomorphism $\tilde{\phi}$ is an extension of the identification of the centers of $\widehat{\C[\tilde{G}_q]}$ and $\widehat{\C[\tilde{G}]}$. Therefore there exists a commutative diagram
\[ \begin{CD}
\widehat{\C[\tilde{G}_q]} @>\tilde{\phi}>> \widehat{\C[\tilde{G}]} \\
 @VVV @ VVV\\
\widehat{\C[G_q]} @>\phi>>  \widehat{\C[G]} \\
\end{CD}
 \]
Since $\tilde{\Fg}_q$ maps tensor products of weight modules to themselves, clearly it restricts to a $\Fg_q \in \prod_{\lambda, \mu\in P_+\cap P_G} End(V_\lambda\otimes V_\mu)$. Thus all properties of $\Fg_q$ and $\phi$ follow from those of $\tilde{\Fg}_q$ and  $\tilde{\phi}$. 
\end{proof}

The element $\Fg_q\in \prod_{\lambda, \mu\in P_+\cap P_G} End(V_\lambda\otimes V_\mu)$ is unitary, so it extends to a bounded multiplier $\Fg_q\in \mathcal{M}(C^*(G)\otimes C^*(G))$. Since it satisfies the cocycle condition the $*$-homomorphism $\Delta_q:=Ad(\Fg_q)\circ \Delta_{C^*(G)}$ defines a comultiplication on $C^*(G)$ which by Theorem $5.3$ of \cite{vandaeledisc} admits invariant Haar weights. Thus, for any compact, semisimple Lie group $G$ we can define the discrete quantum group $C^*(G_q)$ as the $C^*$-algebra $C^*(G)$ with comultiplication twisted by the feasible twist $\Fg_q$:
\[\Delta_q(a):=\Fg_q(\Delta_{C^*(G)}(a))\Fg_q^*.\]
The compact quantum group $C(G_q)$ is defined as the quantum dual of $C^*(G_q)$.

Let us show that the Drinfeld-Jimbo twist is torsion-free. The proof is based on the case of $SU_q(2)$ which was proven in \cite{voigtfree}.  We will include the proof of the case $SU_q(2)$ for the sake of completeness.

\begin{sats}
\label{djtorfree}
If $G$ is a simply connected, compact semisimple Lie group the Drinfeld-Jimbo twists $\hat{G}_q$ are torsion-free for $q>0$. 
\end{sats}

\begin{proof}
To prove that $\hat{G}_q$ is torsion-free we observe that a coaction of $C(G_q)$ on a finite-dimensional $C^*$-algebra $A=\oplus M_{n_j}(\C)$ corresponds to a $*$-action of $\mathcal{U}_q(\mathfrak{g})$ on $A$.  What needs to be proven is that this action is induced by a $*$-representation on $\oplus \C^{n_j}$ in the sense that the $*$-representation of $\mathcal{U}_q(\mathfrak{g})$ on $\oplus \C^{n_j}$ satisfies
\[(x.a)v=\Delta_q(x).(a\otimes v)\]
for $a\in \mathcal{U}_q(\mathfrak{g})$, $a\in A$ and $v\in \oplus \C^{n_j}$. The case $q=1$ follows from that $G$ is simply connected and every projective representation of a simply connected Lie group is a representation. So we can assume that $q\neq 1$. The starting point of our proof is the simplest case $\mathfrak{g}=su(2)$. We let the generators of $\mathcal{U}_q(su(2))$ be denoted by $E,F,K$. 

Assume that the finite-dimensional $C^*$-algebra $A$ has a $*$-action of $\mathcal{U}_q(su(2))$. Consider the restriction of the action to the torus $\T\subseteq SU_q(2)$. Since $\T$ is a connected group the $\T$-action is implemented by conjugating by an invertible self-adjoint matrix $k=\oplus k_j$. We may assume that $k$ only has positive eigenvalues. The action of $\mathcal{U}_q(su(2))$ on $A$ is a bi-algebra action so the two other generators $E$ and $F$ satisfy 
\[E(ab)=E(a)K(b)+aE(b) \quad\mbox{and}\quad F(ab)=F(a)b+K^{-1}(a)F(b)\quad \mbox{for}\;\;a,b\in A.\]
Let $M:=A$ as a left $A$-module and equip $M$ with a right $A$-action by 
\[a.b:=akbk^{-1}=aK(b).\] 
We can view $a\mapsto E(a)$ as Hochschild cocycle $E:A\to M$. Since $A$ is a semisimple algebra the Hochschild cohomology vanishes and there exist an $e\in A$ such that 
\begin{equation}
\label{eett}
E(a)=eak^{-1}-aek^{-1}.
\end{equation}
Similarly we can construct an $f\in A$ such that 
\begin{equation}
\label{feett} 
F(a)=fa-k^{-1}akf.
\end{equation}
Using the relation $KEK^{-1}=q^2E$ it follows that 
\[k(ek^{-1}a-k^{-1}akek^{-1})k^{-1}=kek^{-1}ak^{-1}-akek^{-2}=q^2(ea-ae)k^{-1}.\]
This implies that $c:=kek^{-1}-q^2e$ is central. Replacing $e$ by $e-c(1-q^2)^{-1}$ we obtain the identity $kek^{-1}=q^2e$. This operation clearly preserves equation \eqref{eett}. In the same manner we can redefine 
$f$ so that it satisfies $kfk^{-1}=q^{-2}f$ and equation \eqref{feett}.

To prove the commutation relation $[E,F]=\frac{K-K^{-1}}{q-q^{-1}}$ we will assume that $q\in ]0,1[$. This is no restriction since in Proposition $6$ in Chapter $3$ of \cite{klimsch} an isomorphism of Hopf algebras $\mathcal{U}_q(su(2))\cong\mathcal{U}_{q^{-1}}(su(2))$, leaving the generator $K$ invariant, is constructed. The relation $[E,F]=\frac{K-K^{-1}}{q-q^{-1}}$ in $\mathcal{U}_q(su(2))$ implies that for $a\in A$ 
\begin{align*}
\big(e(fa-k^{-1}akf)k^{-1}&-(fa-k^{-1}akf)ek^{-1}\big)-\\
&-\left(f(eak^{-1}-aek^{-1})-k^{-1}(eak^{-1}-aek^{-1})kf\right)=\\
&=(ef-fe)ak^{-1}-k^{-1}a(ef-fe)=\frac{kak^{-1}-k^{-1}ak}{q-q^{-1}}.
\end{align*}
So there exist a central element $c'$ such that $ef-fe-\frac{k}{q-q^{-1}}=c'k^{-1}$. Since $k$ only has positive eigenvalues, $c'$ must have strictly negative eigenvalues and $\lambda:=\sqrt{c'(q-q^{-1})}$ is a well defined self-adjoint matrix because $q\in ]0,1[$.  After replacing $k$ by $\lambda k$ and $e$ by $\lambda e$ we obtain the relation 
\[ef-fe=\frac{k-k^{-1}}{q-q^{-1}}.\] 
So the $\mathcal{U}_q(su(2))$-action on $A$ lifts to an action on $\oplus \C^{n_j}$. To show that this is a $*$-representation, we observe that since the $\mathcal{U}_q(su(2))$-action is a $*$-action we obtain the relation $(E(a))^*=S(E)^*(a)=-F(a^*)$. Thus for any $a\in A$ we have the equality $a^*(e^*-kf)=(e^*-kf)a^*$. So there exist a central $c''$ such that $e^*-kf=c''$. Conjugating this expression by $k$ we obtain 
\[c''=k(e^*-kf)k^{-1}=q^{-2}e^*-q^{-2}kf=q^{-2}c''.\]
Therefore $c''=0$ and $e^*=kf$.  Thus the $\mathcal{U}_q(su(2))$-action on $\oplus \C^{n_j}$ is a $*$-representation. 

Returning to the general case, assume that we have a $*$-action of $\mathcal{U}_q(\mathfrak{g})$ on $A$. We may as above construct $e_1,\ldots, e_n, f_1, \ldots, f_n, k_1,\ldots, k_n\in A$ satisfying the relations \eqref{treett} and \eqref{tretva} defining the $\mathcal{U}_q(\mathfrak{g})$-action on $A$ as in the equations \eqref{eett} and \eqref{feett}. What remains to be proven is that we can choose the elements $e_1,\ldots, e_n, f_1, \ldots, f_n$ such that the quantum Serre relations \eqref{tretre} and \eqref{trefyr} are satisfied. The Serre elements in $A$ is defined as
\begin{align*}
x_{ij}:=\sum_{k=0}^{1-a_{ij}} (-1)^k
\begin{pmatrix} 1-a_{ij}\\
k \end{pmatrix} _{q_i} &e_i^{1-a_{ij}-k}e_je_i^k
\quad \mbox{and}\\
&y_{ij}:=\sum_{k=0}^{1-a_{ij}} (-1)^k
\begin{pmatrix} 1-a_{ij}\\
k \end{pmatrix} _{q_i} f_i^{1-a_{ij}-k}f_jf_i^k,
\end{align*}
whenever they make sense. If we can prove that $x_{ij}=y_{ij}=0$ the proof of the theorem is complete. We will only prove the identity $x_{ij}=0$, the proof of $y_{ij}=0$ is analogous. A straightforward, but cumbersome, calculation gives the identity
\[\left(\sum_{k=0}^{1-a_{ij}} (-1)^k
\begin{pmatrix} 1-a_{ij}\\
k \end{pmatrix} _{q_i} E_i^{1-a_{ij}-k}E_jE_i^k\right)(a)=[x_{ij},a]k_i^{-(1-a_{ij})}k_j^{-1}\quad \mbox{for} \quad a\in A.\]
Since the left hand side is $0$ the element $x_{ij}$ must be central. However, we have the identities $k_ix_{ij}k_i^{-1}=q_i^{2-a_{ij}}x_{ij}$ and $k_jx_{ij}k_j^{-1}=q_j^{a_{ji}(1-a_{ij})+2}x_{ij}$, so if $a_{ij}\neq -2$ the first identity implies that $x_{ij}=0$ and if $a_{ij}=-2$ the second identity implies $x_{ij}=0$. 
\end{proof}

\end{document}